\documentclass[12pt]{amsart}
\usepackage{amssymb,amsfonts,amsmath}
\usepackage{mathrsfs}
\usepackage{anysize}
\usepackage{url}
\usepackage[leqno]{amsmath}
\usepackage{enumitem}
\usepackage{tikz-cd}

\theoremstyle{plain}
\newtheorem{theorem}{Theorem}[section]
\newtheorem{lemma}[theorem]{Lemma}
\newtheorem{corollary}[theorem]{Corollary}

\theoremstyle{definition}

\theoremstyle{remark}

\usetikzlibrary{decorations.markings}
\tikzset{negated/.style={
        decoration={markings,
            mark= at position 0.5 with {
                \node[transform shape] (tempnode) {${\scriptstyle\setminus} $};
            }
        },
        postaction={decorate}
    }
}

\tikzset{degil/.style={
            decoration={markings,
            mark= at position 0.5 with {
                  \node[transform shape] (tempnode) {$\backslash$};
                  }
              },
              postaction={decorate}
}
}

\title[Symbolic dynamics of piecewise contractions]{Symbolic dynamics of piecewise contractions}


\medskip
\begin{document}

\maketitle

\centerline{\scshape Benito Pires}

{\footnotesize
	\centerline{Departamento de Computa\c c\~ao e Matem\'atica, Faculdade de Filosofia, Ci\^encias e Letras}
	\centerline {Universidade de S\~ao Paulo, 14040-901, Ribeir\~ao Preto - SP, Brazil}
	\centerline{benito@usp.br} } 

\bigskip

\marginsize{2.5cm}{2.5cm}{1cm}{2cm}

\begin{abstract} 
A map $f{:}\,[0,1)\to [0,1)$ is a {\it piecewise contraction of $n$ intervals} ($n$-PC) if there exist $0<\lambda<1$ and a partition of $I=[0,1)$ into intervals $I_1,I_2,\ldots,I_n$  such that $\left\vert f(x)-f(y)\vert\le \lambda\vert x-y\right\vert$ for every $x,y \in I_i$ ($i=1,2,\ldots,n$). An infinite word $\theta=\theta_0\theta_1\ldots$ over the alphabet $\mathcal{A}=\{1,\ldots,n\}$ is a {\it natural coding of} $f$ if there exists $x\in I$ such that
$\theta_k=i$  whenever $f^k(x)\in I_i$. We prove that if $\theta$ is a natural coding of an injective $n$-PC, then some infinite subword of $\theta$
  is either periodic or isomorphic to a natural coding of a topologically transitive $m$-interval exchange transformation ($m$-IET), where $m\le n$. Conversely, every natural coding of a topologically transitive $n$-IET is also a natural coding of some injective $n$-PC.\end{abstract}

\maketitle

\section{Introduction}

Throughout this article, let $I=[0,1)$ denote the unit interval. A map $f{:}\,I\to I$ is a {\it piecewise contraction of $n$ intervals} ($n$-PC) if there exist $0< \lambda <1$ and a partition of $I$ into non-degenerate intervals $I_1,\ldots,I_n$ such that $f\vert_{I_i}$ is $\lambda$-Lipschitz for every $1\le i\le n$. If, in particular, there exist $b_1,\ldots,b_n\in \mathbb{R}$ and $\sigma_1,\ldots,\sigma_n\in\{-1,1\}$ such that $f(x)=\sigma_i\lambda x+b_i$ for every $x\in I_i$, then we say that $f$ is a {\it piecewise $\lambda$-affine contraction}. 

The {\it natural $f$-coding of a point $x\in I$} is the infinite word $\theta_f(x)=\theta_0 \theta_1\ldots$ defined by $\theta_k=i$  whenever $f^k(x)\in I_i$, where $f^0$ denotes the identity map. 
We say that an infinite word $\theta$ is a {\it natural coding of $f$} if $\theta=\theta_f(x)$ for some $x\in I$. We say that $\theta$ is \textit{ultimately periodic} (respectively, \textit{periodic}) if there exist finite subwords $u,v$ of $\theta$ such that $\theta=uvv\ldots$ (respectively, $\theta=vv\ldots$).
The {\it language $\mathcal{L}(\theta)$} of a natural coding $\theta$ is the union of the sets
$L_k(\theta)=\{\theta_{m}\theta_{m+1}\cdots\theta_{m+k-1}:m\ge 0\}$ of finite subwords of length $k$ occuring in $\theta$, where $L_0$ is the one-point-set formed by the empty word. 

In this article, we give a complete and systematic description of the languages of injective $n$-PCs, $n\ge 2$, by providing a dictionary between these languages and the fairly well-understood languages of interval exchange transformations (IETs). We also provide converse results which enable us to construct $n$-PCS with any prescribed admissible coding.

The first point addressed in this article consists in providing the list of all admissible natural codings of injective $n$-PCs. Natural codings of piecewise contractions defined on $2$ intervals (or more generally, defined on $2$ complete metric spaces) were provided by Gambaudo and Tresser \cite{GT1} and are intrinsically related to natural codings of  rotations of the circle. Concerning languages of injective $n$-PCs $f{:}\,I\to I$ for $n>2$, some progress was made recently by Catsigeras, Guiraud and Meyroneinc  \cite{CGM}. They proved that for each natural coding $\theta$ of $f$, the complexity function of the language
$\mathcal{L}(\theta)$, defined by $p_{\theta}(k)=\#L_k(\theta)$, where $\#$ denotes cardinality, is eventually affine. 

The second point concerns the problem of how to construct $n$-PCs with any prescribed list of admissible natural codings.
In this regard, it follows from the works \cite{NP,NPR1,NPR2} that a generic $n$-PC admits only ultimately periodic natural codings. Therefore, $n$-PCs with no ultimately periodic natural coding are exotic and their construction is a nontrivial issue.
 The existence of $2$-PCs having no ultimately periodic natural coding is related to the existence of smooth flows on the $2$-torus with pathological dynamics (see Cherry \cite{Cherry}). More generally, $2$-PCs topologically semiconjugate to irrational rotations have being constructed and studied via a rotation number approach (see \cite{B1,B2,BC,GT2,JO,LNo}). Here we address the second point in full generality, by using another approach, based on the existence of an invariant measure (see \cite{BP}). In particular, we prove that every minimal $n$-IET, with $n\ge 2$, with or without flips, is a topological factor of an $n$-PC with no ultimately periodic natural coding. This combined with Keane's irrationality criteria \cite[p. 27]{MK} provides a huge class of exotic $n$-PCs. Since every irrational rotation can be considered as a minimal $2$-PC, the previous results fit into our framework.   
 
 As for the motivation to study $n$-PCs, it is worth remarking that they describe pretty well the dynamics of some Cherry flows on $2$-manifolds, dissipative outer billiards, traffic systems, queueing systems and switched server systems.\\
 
 \noindent{\textbf{Acknowledgments.}} The author is very grateful to Filipe Fernandes and Francisco Braun whose comments  contributed to the improvement of the first version. The author was partially supported by grant \# 2018/06916-0, S\~ao Paulo Research Foundation (FAPESP) and by CNPq.
  
\section{Statement of the results}

A bijective map $T{:}\,I\to I$ is an {\it $n$-interval exchange transformation} ($n$-IET) if there exist a partition $I_1,\ldots,I_n$ of $I$ into non-degenerate intervals,  $b_1,\ldots,b_n\in \mathbb{R}$ and $\sigma_1,\ldots,\sigma_n\in\{-1,1\}$ such that $T(x)=\sigma_i x+b_i$ for every $x\in I_i$. an $n$-IET $T$ is {\it standard} if $I_i$ a left-closed right-open interval and $T\vert_{I_i}$ is the translation $x\mapsto x+b_i$ for every $1\le i\le n$.  
Following \cite{N}, we say that a non-standard $n$-IET $T$ {\it has flips} if, for some $1\le i\le n$, $T\vert_{I_i}$ is the map $x\mapsto -x+b_i$. We say that an $n$-IET $T{:}\,I\to I$ is \textit{irreducible} if there is no $0<\delta<1$ such that
$T\big([0,\delta)\big)\subset [0,\delta)$ or, equivalently, if there is no $1\le j\le n-1$ such that $T\big(\cup_{i=1}^j I_i\big)\subset \cup_{i=1}^j I_i$.

An $n$-IET $T{:}\,I\to I$ is: {\it topologically transitive} if it has a dense $T$-orbit $\left\{x,T(x),\ldots\right\}$, {\it minimal} if every $T$-orbit is dense, and {\it aperiodic} if it has no periodic orbit. A periodic orbit $\gamma$ is {\it attractive} if there exists an open set $U_{\gamma}$  such that $\omega(x)=\gamma$ for every $x\in U_{\gamma}$. In a standard $n$-IET, every periodic orbit is attractive, thus in this case topological transitivity is equivalent to minimality. An $n$-IET $T$ satisfies the {\it infinite distinct orbit condition} (i.d.o.c.) if the orbits $\{x_i,T(x_i),\ldots\}$
 of its discontinuities $x_i$, $1\le i\le n-1$, are infinite and pairwise disjoint. Keane  \cite{MK} proved that every irreducible standard $n$-IET, $n\ge 2$, satisfying the i.d.o.c. is {\it minimal}. The {\it natural $T$-coding} of a point $x\in I$ is the infinite word $\theta_T(x)=\theta_0\theta_1\ldots$ defined by $\theta_k=i$ whenever $T^k(x)\in I_i$. If $T$ is an irreducible standard $n$-IET satisfying the i.d.o.c., then the language $\mathcal{L}(\theta)$ of a natural $T$-coding $\theta$ is the same for any $\theta$. In this case, we define the {\it language of $T$}, denoted by $\mathcal{L}$, to be the language $\mathcal{L}(\theta)$ of any of its natural $T$-codings.
  
The {\it alphabet} $\mathcal{A}(\theta)$ of an infinite word $\theta=\theta_0 \theta_1\ldots$ is the set of letters that occur in $\theta$. We say that two infinite words $\theta=\theta_0 \theta_1\ldots$ and $\omega=\omega_0\omega_1\ldots$ are {\it isomorphic} if their alphabets
$\mathcal{A}({\theta})$ and $\mathcal{A}({\omega})$ have the same cardinality and there is a bijection $\pi{:}\,\mathcal{A}({\theta})\to\mathcal{A}({\omega})$ such that $\omega_k=\pi(\theta_k)$ for every $k\ge 0$. For example, the  infinite words
$$
\theta = 010010001\ldots \quad\textrm{and}\quad  \omega = 121121112\ldots
$$
are isomorphic because $\mathcal{A}(\theta)=\{0,1\}$ and $\mathcal{A}(\omega)=\{1,2\}$ have the same cardinality and
$\omega_k=\pi(\theta_k)$ for every $k\ge 0$, where the bijection $\pi:\mathcal{A}(\theta)\to \mathcal{A}(\omega)$ is given by 
$\pi(0)=1$ and $\pi(1)=2$. 

Our main results are the following.

\begin{theorem}\label{thm1} Let $f{:}\,I\to I$ be an injective $n$-PC, then there exist $2\le m\le n$ and an $m$-IET $T{:}\,I\to I$ without attractive periodic orbits such that for each $x\in I$ there exists an integer $k\ge 0$ such that the natural $f$-coding of $f^k(x)$  is either  periodic or 
isomorphic to a non ultimately periodic natural coding of $T$.
\end{theorem} 

\begin{theorem}\label{thm2} Given any topologically transitive $n$-IET $T{:}\,I\to I$, there exist an injective  piecewise $\frac12$-affine contraction $f_T{:}\,I\to I$ of $n$ intervals
and a continuous, surjective, nondecreasing map $h{:}\,I\to I$ such that $\theta_{f_T}(x)=\theta_T\big(h(x)\big)$ for every $x\in I$. In particular, if $T$ is an irreducible standard $n$-IET satisfying the i.d.o.c., then
 the language of each natural coding of $f_T$ equals the language of $T$. 
\end{theorem} 

In Theorem \ref{thm1}, the term ``without attractive periodic orbits" may be replaced by ``aperiodic" in the case in which $I_i=\big[x_{i-1},x_i\big)$ and $f\vert_{I_i}$, $1\le i\le n$, is (strictly) increasing, where $I_1,\ldots,I_n$ is the partition associated to $f$. 

Theorem \ref{thm1} turns out to be a dictionary between languages of PCs and languages of IETs. Languages of minimal  IETs were studied by Belov and Chernyat'ev \cite{ABALC},  Ferenczi \cite{SF1}, Ferenczi and Zamboni \cite{FZ}, and Dolce and Perrin \cite{DP}. In particular, it is known that if $\theta$ is a natural coding of an irreducible standard minimal $n$-IET, then the language $\mathcal{L}(\theta)$ 
 does not depend on $\theta$, is uniformly recurrent and has complexity function satisfying $p_{\theta}(k)\le (n-1)k+1$, where the equality holds  if the IET satisfies the i.d.o.c.. Languages generated by substitutions (e.g. languages of self-similar IETs) were studied by Lopez and Narbel \cite{LN}. 
 Natural codings of aperiodic $n$-IETs are isomorphic to natural codings of topologically transitive $n$-IETs.
 
 Theorem \ref{thm2} provides examples of $n$-PCs without periodic or ghost orbits, or equivalently, without ultimately periodic natural codings. These examples are not easy to construct because generically $n$-PCs of the interval are asymptotically periodic (see \cite{NPR1,NPR2}). The following result is a corollary of Theorem \ref{thm1}.  
 
\begin{corollary}\label{cor0} Let $\theta$ be a natural coding of an injective $n$-PC, then some infinite subword of $\theta$
  is either periodic or isomorphic to a non ultimately periodic natural coding of a topologically transitive $m$-IET, where $2\le m\le n$.
\end{corollary}

Theorems \ref{thm1} and \ref{thm2} imply, in particular, the result of Catsigeras, Guiraud and Meyroneinc \cite{CGM} concerning the complexity function of languages of $n$-PCs, which is stated below in a more complete way, with $f_T$ given by Theorem $\ref{thm2}$.
\begin{corollary}\label{cor1} 
Let $\theta$ be a natural coding of an injective $n$-PC $f{:}\,I\to I$, then 
\begin{itemize}
\item [$(i)$]
There exist $\alpha\in \{0,1,\ldots,n-1\}$, $\beta\ge 1$ and $k_0\ge 1$ such that the complexity function of $\theta$ satisfies $p_{\theta}(k)=\alpha k + \beta$ for every $k\ge k_0$ with $\beta=1$ if $\alpha=n-1$;
\item [$(ii)$] If $n\ge 2$, $T{:}\,I\to I$ is a standard $n$-IET satisfying the i.d.o.c. and $f=f_T$, then $p_{\theta}(k)=(n-1)k+1$ for every $k\ge 1$.
\end{itemize}
\end{corollary}

The particular family of $2$-PCs $f{:}\,I\to I$ defined by $f(x)=\lambda x+\delta\,\, (\textrm{mod}\, 1)$ was considered by Bugeaud \cite{B1,B2}, Bugeaud and Conze \cite{BC} and, more recently,  by Janson and \"Oberg \cite{JO}, and also by Laurent and Nogueira \cite{LNo}, by means of a rotation number approach. Concerning such family, we provide the following corollary, which turns out to be a special case of \cite[Corollary 7]{LNo}. We recall that an $n$-PC $f{:}I\to I$ is {\it topologically semiconjugate} to an $n$-IET $T{:}I\to I$ if there exists a
continuous, nondecreasing and surjective map $h{:}I\to I$ such that $h\circ f=T\circ h$.

\begin{corollary}\label{cor2} For each irrational $0<\alpha<1$, there exists a transcendent $\delta\in\mathbb{R}$ such that
the $2$-PC $f_T{:}\,I\to I$ and the minimal $2$-IET $T{:}\,I\to I$ defined by
$$
f_T(x)=\frac12 x +\delta\,\,({\rm mod}\,\,1)\quad\textrm{and}\quad T(x)= x +\alpha\,\,({\rm mod}\,\,1)
$$
are topologically semiconjugate and every natural coding of $f_T$ is a Sturmian sequence. In particular,
if $\alpha=2-\varphi$, where $\varphi=(1+\sqrt{5})/2$ is the golden ratio,
then $\delta=1-\frac{R}{2}$, where $R$ is the rabbit constant.
\end{corollary}

In Corollary \ref{cor2}, we have that $\delta=\frac14\sum_{k\ge 0} \theta_{k} 2^{-k}$, where $\theta=\theta_0\theta_1\ldots$ is the natural coding of $\alpha$ under the action of the irrational rotation $x\mapsto x+\alpha\,\,({\rm mod}$\,\,1).

\section{Preparatory lemmas}

In this section, we present some results that will be used to prove Theorem \ref{thm1}, as the next table  clarifies.

\begin{table}[h]
\begin{tabular}{  l |  l }
  \hline			
  result & needs \\ \hline
   Lemma \ref{lemathree} & Lemma \ref{lematwo} \\
   Corollary \ref{cor35} & Lemma \ref{lemathree} \\
   Theorem \ref{proa} & Lemma \ref{lemafour} and \cite[Theorem 2.1]{BP} \\
   Theorem \ref{thm1} & Lemma \ref{lemaone}, Corollary \ref{cor35} and Theorem \ref{proa}   \\
  \hline  
\end{tabular}
\end{table}

Throughout this section, let $I=[0,1)$ and  $f{:}\,I\to I$ be an injective $n$-PC with associated partition $I_1,\ldots,I_n$ whose endpoints are
$0=x_0<x_1<\cdots< x_n=1$. 


The {\it $\omega$-limit set} of $x\in I$ is defined by
$$\omega(x)=\bigcap_{\ell\ge 1}\overline{\bigcup_{k\ge \ell} \left\{ f^k(x)\right\}}\,,$$
where $\overline{S}$ denotes the topological closure in $\mathbb{R}$ of any set $S\subset I$.

\begin{lemma}\label{lemaone} Let $x\in I$ be such that $\omega(x)$ is finite, then there exists an integer $k\ge 0$ such that
the natural $f$-coding of $f^k(x)$ is periodic.
\end{lemma}
\begin{proof}  We may assume that $\{x,f(x),f^2(x),\ldots\}$ is an infinite set, otherwise 
 $x$ would be a periodic point, then we could take $k=0$. Since $\omega(x)$ is a finite set, we may write $\omega(x)=\{p_1,\ldots,p_r\}$. Without loss of generality, we suppose that
$\omega(x)\subset (0,1)$, thus there exists $\epsilon>0$ so small that
$$ \epsilon<\frac14 \min_{1\le i<j\le r} \vert p_i-p_j\vert\,\,\textrm{and}\,\,  
\bigcup_{j=1}^r (p_j-\epsilon,p_j)\cup (p_j,p_j+\epsilon)\subset I{\setminus} \{x_0,x_1,\ldots,x_{n-1}\}.$$
In particular, if
$$ \mathscr{I}=\left\{(p_1-\epsilon, p_1),(p_1,p_1+\epsilon),\ldots,(p_r-\epsilon, p_r),(p_r,p_r+\epsilon)\right\},
$$
then $f(J)$ is an open interval for every $J\in\mathscr{I}$. 

Let $\mathscr{I}'\subset \mathscr{I}$ denote the subcollection formed by the intervals that are visited infinitely many times by the $f$-orbit of $x$, that is,
$$ \mathscr{I}'=\left\{ J\in\mathscr{I}: \{x,f(x),f^2(x),\ldots\}\cap J \,\,\textrm{is an infinite set}\right\}.
$$
We claim that for each $J_1\in \mathscr{I}'$, there exists ${J_2}\in \mathscr{I}'$ such that
$f(J_1)\subset J_2$. Without loss of generality, suppose that $J_1=(p-\epsilon,p)$, where $p\in \omega(x)$. As $J_1\subset I{\setminus}\{x_0,x_1,\ldots,x_{n-1}\}$, we have that $f\vert_{J_1}$ is a contraction, thus $f(J_1)$ is an open interval of length smaller than
$\epsilon$. On the other hand, since $J_1\in\mathscr{I}'$, there exists an increasing sequence of integers $0\le k_1<k_2<\cdots$ such that $\left\{f^{k_j}(x)\right\}_{j\ge 1}\subset J$. Notice that
 $\lim_{j\to\infty} f^{k_j}(x)=p$, otherwise there would exist a point of $\omega(x)$ in $J_1$ different from $p$, which contradicts the first inequality in the definition of $\epsilon$. Because $f\vert_{J_1}$ is injective and continuous, we have that $\left\{f^{k_j+1}(x)\right\}_{j\ge 1}\subset f(J_1)$ converges to some point  $q\in\omega(x)\cap\partial f(J_1)$, where $\partial f(J_1)$ denotes the endpoints of the open interval $f(J_1)$. Putting it all together, we conclude that $f(J_1)$ is an open interval
 that contains infinitely many points of the $f$-orbit of $x$, has length smaller than $\epsilon$, and has an endpoint in $\omega(x)$. Therefore, there exists $J_2\in\mathscr{I}'$ such that $f(J_1)\subset J_2$.

To finish the proof, let $J\in\mathscr{I}'$, then there exists $k'\ge 0$ such that $f^{k'}(x)\in J$. By the claim,
there exist $1\le i_1<i_2$ and intervals $J_1,\ldots,J_{i_1},J_{i_1+1},\ldots,J_{i_2}\in\mathscr{I'}$ such that $J_1$=J, $J_{i_1}=J_{i_2}$ and $f(J_{i})\subset J_{i+1}$ for all $1\le i\le i_2-1$, proving that $f^k(x)$ has a periodic natural $f$-coding for some $k\ge k'$.
\end{proof}

\begin{lemma}\label{lematwo} Let $J\subset I$ be an open interval, then there exists a finite set $B\subset I$ such that if $J_0\subset J{\setminus}B$ is an open interval, then one of the following happens:
\begin{itemize}
\item [$(i)$] $f(J_0)$, $f^2(J_0)$, \ldots are pairwise disjoint open intervals contained in $I{\setminus} J$;
\item [$(ii)$] $\exists m\ge 0$ such that $f^{m+1}(J_0)$ is open subinterval of $J$. Moreover, 
  if $m\ge 1$, then $f(J_0),\ldots,f^m(J_0)$ are open subintervals of $I{\setminus} \big(J\cup \{x_0,x_1,\ldots,x_{n-1}\})$. 
\end{itemize}
\end{lemma}
\begin{proof} Let $J\subset I$ be an open interval. Given $x\in I$, set
\begin{equation}\label{trenta}
\tau_x=\min \,\left\{k\ge 0: f^{-k}\big(\{x\}\big)\subset J \right\},
\end{equation}
where by convention $\inf\emptyset=\infty$. Let
$$B=\bigcup\left\{f^{-\tau_x}(\{x\}): x\in \{x_0,x_1,\ldots,x_{n-1}\}\cup \partial J \,\,\textrm{and}\,\, \tau_x<\infty\right\}.$$
Let $J_0\subset J{\setminus }B$ be an open interval, then one of the following alternatives happens:  $\{f^k(J_0)\}_{k\ge 1}\subset I{\setminus }J$ or  there exists an  integer $\ell\ge 1$ such that $f^\ell(J_0)\cap J\neq \emptyset$. In the first case, by the injectivity of $f$ and also because  $J_0\subset J{\setminus B}$, we have that
$J_0$, $f(J_0)$, $f^2(J_0)$,\ldots are pairwise disjoint sets contained in $I{\setminus} \{x_0,x_1,\ldots,x_{n-1}\}$. Since each $f\vert_{I_i}$ is Lipschitz continuous, we conclude that $f^k(J_0)$ is an open interval for every $k\ge 0$, which proves $(i)$.
As for the second alternative, let $m=\min\,\left\{\ell\ge 1: f^\ell(J_0)\cap J\neq\emptyset\right\}-1$. If $m=0$, then $f(J_0)\cap J \neq\emptyset$, which together with the fact that $J_0\subset J{\setminus B}$ implies that
$f(J_0)$ is an open subset of $J$. Otherwise, if $m\ge 1$, then proceeding as in the first case yields that the sets $f(J_0),\ldots, f^{m}(J_0)$ are pairwise disjoint open subintervals of $I{\setminus} \big(J\cup \{x_0,x_1,\ldots,x_{n-1}\})$. Moreover, because $J_0\subset J{\setminus B}$, we have that   $f^{m+1}(J_0)\cap J \neq\emptyset$ implies  that $f^{m+1}(J_0)$ is an open subinterval of $J$.
 \end{proof}
\noindent \textbf{Remark}. The item $(ii)$ of Lemma \ref{lematwo} implies that $f^{m+1}\vert_{J_0}:J_0\mapsto f^{m+1}(J_0)$ is a bijective contraction.
 
  \begin{lemma}\label{lemathree} If for some $x\in I$ and $1\le i\le n$, the set
$\left\{x,f(x),f^2(x),\ldots\right\}\cap (x_{i-1},x_i)$ is infinite and $\omega(x)\cap (x_{i-1},x_i)=\emptyset$, then
$\omega(x)$ is finite.
\end{lemma}
\begin{proof}
By hypothesis, we have that
\begin{itemize} \item [(H1)] $\left\{x,f(x),f^2(x),\ldots\right\}\cap (x_{i-1},x_i)$ is infinite;
\item [(H2)] $\left\{x,f(x),f^2(x),\ldots\right\}\cap K$ is finite for all compact set $K\subset (x_{i-1},x_i)$.
\end{itemize} 
By (H1), the orbit of $x$ returns to $J=(x_{i-1},x_i)$ infinitely many times. Let $1\le k_1<k_2<\cdots$ denote the return times of $x$ to $J$ under the action of $f$. Because of (H2), we have only three cases to consider.\\

\indent Case (a). $\lim_{j\to\infty} f^{k_j}(x)=x_{i-1}$. \\

Let $B$ the finite set given by Lemma \ref{lematwo} considering $J=(x_{i-1},x_i)$. Let $\epsilon>0$ be so small that
$J_0=(x_{i-1},x_{i-1}+\epsilon)$ is a subset of $J{\setminus} B$. Notice that the alternative $(i)$ of Lemma \ref{lematwo} cannot occur. In fact, since $f^{k_{j}}(x)\downarrow x_{i-1}$, we have that $f^k(J_0)\cap J\neq\emptyset$ for many positive values of $k$. By exclusion, the item $(ii)$ of Lemma \ref{lematwo} is true, then there exists $m\ge 0$ such that
$f^{m+1}(J_0)$ is an open subinterval of $J$ and, if $m\ge 1$, then
 $f(J_0),\ldots,f^m(J_0)$ are open subintervals of $I{\setminus} \big(J\cup \{x_0,x_1,\ldots,x_{n-1}\})$. In particular, if $y\in J_0$, then $m+1$ is  the first return time of $y$ to $J$. This means that if $j_0\ge 1$ is such that $\{f^{k_j}(x)\}_{j\ge  j_0}\subset J_0$, then $\{f^{k_j}(x)\}_{j> j_0}\subset f^{m+1}(J_0)$, implying that 
 $x_{i-1}$ belongs to the boundary of the open interval $f^{m+1}(J_0)$. Moreover, since $f^{m+1}\vert_{J_0}{:}J_0\to f^{m+1}(J_0)$ is a bijective contraction (see the Remark after Lemma \ref{lematwo}), we have that
$f^{m+1}(J_0)\subset J$ is an open interval with length smaller than $\epsilon$ and with an endpoint in $x_{i-1}$, thus
$f^{m+1}(J_0)\subset (x_{i-1},x_{i-1}+\epsilon)=J_0$. This implies that 
$\omega(x)$ is finite.\\

\indent Case (b). $\lim_{j\to\infty} f^{k_j}(x)=x_{i}$.\\

Just proceed as in Case (a) considering now $J_0=(x_{i}-\epsilon,x_{i})$. \\

\indent Case (c). $\bigcap_{\ell\ge 1} \overline{\bigcup_{j\ge \ell} \big\{f^{k_j}(x)\big\}}=\{x_{i-1},x_i\}.$
\\

The proof presented here  is a variation of that used in Case (a). Let $\epsilon>0$ be so small that $J_0'=(x_{i-1},x_{i-1}+\epsilon)$ and
$J_0'' =(x_{i}-\epsilon,x_{i})$ are contained in $J{\setminus} B$. Then by the same arguments used in Case (a), there exist $m', m''\ge 0$ such that $f^{m'+1}(J_0')$ and $f^{m''+1}(J_0'')$ are disjoint open subintervals of $J$ and, if $m'\ge 1$ (respectively, $m '' \ge 1$), then
$f(J_0'),\ldots,f^{m'}(J_0')$ \big(respectively, $f(J_0''),\ldots,f^{m''}(J_0'')$\big) are open subintervals of $I{\setminus} \big(J\cup \{x_0,x_1,\ldots,x_{n-1}\})$. In particular, if $y\in J_0'$ (respectively, if $y\in J_0''$), then $m'+1$ (respectively, $m''+1$) is the first return time of $y$ to $J$. This means that if $j_0\ge 1$ is such that 
$\{f^{k_j}(x)\}_{j\ge j_0}\subset J_0'\cup J_0''$, then $\{f^{k_j}(x)\}_{j> j_0}\subset f^{m'+1}(J_0')\cup f^{m''+1}(J_0'')$, implying that $x_{i-1}\in\partial f^{m''+1}(J_0'')$ and $x_{i}\in\partial f^{m'+1}(J_0')$. Moreover, since $f^{m'+1}\vert_{J_0'}{:}J_0'\to f^{m'+1}(J_0')$ and $f^{m''+1}\vert_{J_0''}{:}J_0''\to f^{m''+1}(J_0'')$ are bijective contractions, we can argue in the same way as in Case (a) to conclude that $f^{m'+1}(J_0')\subset J_0''$ and $f^{m''+1}(J_0'')\subset J_0'$, proving that $\omega(x)$ is finite.
\end{proof}
Lemma \ref{lemathree} leads to the following result.
\begin{corollary}\label{cor35} Let $x\in I$ be such that $\omega(x)$ is infinite. If for some $1\le i\le n$, the set
$\left\{x,f(x),f^2(x),\ldots\right\}\cap (x_{i-1},x_i)$ is infinite, then $\omega(x)\cap (x_{i-1},x_i)\neq\emptyset$.
\end{corollary}

We will also need the following result, which is a variation of
 \cite[Theorem 2.1]{BP}.
 
\begin{theorem}\label{proa} Let $x\in I$ be such that $\Lambda=\omega(x)$ is infinite, then there exists a non-atomic $f$-invariant Borel probability measure whose support is $\Lambda$.
\end{theorem}

The proof of Theorem \ref{proa} depends on Lemma \ref{lemafour} stated below. In what follows, let $x\in I$ be such that $\Lambda=\omega(x)$ is infinite. As $x$ is not periodic, there exists $\ell\ge 0$ such that $\{f^k(x):k\ge \ell\}\cap \{x_0,x_1,\ldots,x_{n-1}\}=\emptyset$. Hence, by replacing $x$ by $f^{\ell}(x)$ if necessary, we  assume that
 \begin{equation}\label{naa}
 \left\{x,f(x),f^2(x),\ldots \right\}\cap \{x_0,x_1,\ldots,x_{n-1}\}=\emptyset.
 \end{equation}
Denote by $\{\nu_m\}_{m\ge 1}$ the sequence of Borel probability measures on $I$ defined by 
$$
\nu_m=\dfrac{1}{m}\sum_{k=0}^{m-1} \delta_{f^{k}(x)},
$$
where $\delta_{f^k(x)}$ is the Dirac probability measure on $I$ concentrated at $f^k(x)$. By the
Banach-Alaoglu Theorem, there exist a Borel probability measure on $I$, denoted henceforth by $\nu$, and a subsequence of  $\{\nu_m\}_{m\ge 1}$, denoted henceforth by $\{\nu_{m_j}\}_{j\ge 1}$, that converges to $\nu$ in the weak$^\star$- topology. We will keep these notations until the end of this section.

\begin{lemma}\label{lemafour} Let $y\in I$, then there exist an open subinterval $J_y$ of $I$ containing $y$ and an integer $j_0\ge 1$ such that $\nu_{m_j}(J_y)<\epsilon$ for every $j\ge j_0$. Moreover, the support of $\nu$ is $\Lambda=\omega(x)$.
\end{lemma} 
\begin{proof} Let $y\in I$ and $\epsilon>0$. We will prove that there exist $\delta>0$ and $j_0\ge 1$ such that the interval
$$ J_y=\begin{cases}
[0,\delta) & \textrm{if}\,\, y=0 \\
(y-\delta,y+\delta) & \textrm{if} \,\, y>0
\end{cases}
$$
satisfies $J_y\subset I$ and $\nu_{m_j}\left(J_y\right)<\epsilon$ for all $j\ge j_0$. Without loss of generality, we may assume that $y>0$ and $J_y=(y-\delta,y+\delta)$. Since $\nu$ is a probability measure, $\nu$ has at most countably many atoms, which means that the set
 $$\Delta=\left\{0<\delta<\min\, \{y,1-y\}: \nu\big(\{y-\delta,y+\delta\}\big)=0 \right\}$$ 
 contains arbitrarily small values of $\delta$. It follows from \cite[Theorem 6.1, p. 40]{P} that if
   $\delta\in \Delta$, then 
 \begin{equation}\label{x49}
 J_y\subset I\quad\textrm{and}\quad \nu(J_y)=\lim_{j\to\infty} \nu_{m_j}(J_y).
 \end{equation}

Now have two cases to consider. \\

Case I\,:  $y\not\in\Lambda$, that is, $y\not\in\omega(x)$. 

In this case, there exist $\delta\in\Delta$ and $j_0\ge 1$ such that
 $f^k(x) \not\in J_y$ for every $k\ge m_{j_0}$. Let $j_1\ge j_0$ be such that
  $m_{j}>\ m_{j_0}/\epsilon$ for every $j\ge j_1$, then
$$\nu_{m_j}(J_y)=\dfrac{1}{m_j}\# \left\{0\le k\le m_{j_0}-1: f^k(x)\in J_y\right\}\le \frac{m_{j_0}}{m_j}<\epsilon,\quad \forall j\ge j_1.$$
Moreover, making $j\to\infty$ and using $(\ref{x49})$ yield $\nu(J_y)=0$, implying that $y$ does not belong to the support of $\nu$.\\

Case II\,: $y\in\Lambda$. \\

First assume that 
there exists an increasing sequence  of integers  $1\le k_1<k_2<\cdots$ such that
$ f^{k_j}(x)\uparrow y$. Since $f$ is an injective piecewise contraction, the following limits are well-defined:
$$
 y_0=y, \quad y_1=\lim_{j\to\infty} f\big(f^{k_j}(x)\big), \quad y_2=\lim_{j\to\infty} f^2\big(f^{k_j}(x)\big),\quad \ldots$$
 We claim that $\#\{k\ge 1: y_k=y\}\le 1$. By way of contradiction, suppose that there exist $1\le p<q$ such that
 $y_{q}=y_{p}=y$. It is elementary to see that for every $\delta>0$  small enough and $A_0=(y-{\delta},y)$, the sets $A_1=f(A_0),A_2=f^2(A_0),\ldots, A_{q}=f^{q}(A_0)$ are open intervals of length less than $\delta$. Yet, $y_k\in \partial A_k$ for every $0\le k\le q$. Hence, either $A_{p}\subset A_0$ or $A_{q}\subset A_0$, which contradicts the fact that $\omega(x)$ is infinite. In this way, the claim is true. Then, there exists $r_0\ge 1$ such that $y_k\neq y$ for all $k\ge r_0$. In particular, given $r\ge 1$, there exists $\delta_1=\delta_1(r)$ such that
 for every $0<\delta<\delta_1$, 
  $$  \#\left\{0\le k\le r-1: f^k\big((y-\delta,y)\big)\cap J_y\neq\emptyset   \right\}\le 2.$$
  Let $r>0$ be such that $\frac{3}{r}<\frac{\epsilon}{3}$. Set $\delta_1=\delta_1(r)$. Then, for all $0<\delta<\delta_1$ with $\delta\in\Delta$ and for any $j$ large enough,
 $$\nu_{m_j}\big((y-\delta,y)\big)=\dfrac{1}{m_j}\# \left\{0\le k\le m_{j}-1: f^k(x)\in J_y\right\}\le \frac{3}{r}<\frac{\epsilon}{3}.$$
 Now assume that the sequence $1\le k_1<k_2<\dots$ does not exist, then for every $\delta$ small enough,
 $$ \nu_{m_j}\big((y-\delta,y\big))=0<\frac{\epsilon}{3}.$$
 Likewise, there exists $\delta_2>0$ such that for all $0<\delta<\delta_2$ with $\delta\in\Delta$, we have that
  $\nu_{m_j}\big((y,y+\delta)\big)<\frac{\epsilon}{3}$ for any $j$ large enough.
 Moreover, $\nu_{m_j}(\{y\})<\frac{\epsilon}{3}$ for any $j$ large enough. Putting all together, there exist
 $\delta>0$ with $\delta\in\Delta$ and $j_0\ge 1$ such that $\nu_{m_j}(J_y)<\epsilon$ for all $j\ge j_0$.
 
It remains to prove that in this case $y$ belongs to the support of $\nu$. By the above, we know that the orbit of $x$ enters in $J_y$ infinitely many times. If we prove that the return times of $x$ to $J_y$ are bounded, then we will conclude that $\inf_{j\ge j_0} \nu_{m_j}(J_y)>0$, which together with $(\ref{x49})$ will imply that $\nu(J_y)>0$. Let $S=\{x_0,x_1,\ldots,x_{n-1}\}\cup\partial J_y$ and $S'=\left\{z\in S: \cup_{k\ge 0} f^{-k}(\{z\})\cap J_y\neq \emptyset \right\}$. Given $z\in S'$, let $$\tau_z=\min\, \{k\ge 0:  f^{-k}(\{z\})\subset J_y\}$$ and
$B=\{f^{-\tau_z}(z): z\in S'\}$. If $U$ is a connected component of $J_y{\setminus B}$, then all points of $U$ either never return to $J_y$ or return to $J_y$ at the same time. The second case always happens when $U\subset J_y$ is a small interval with an endpoint at $y$. In particular, the return times of the points of the orbit of $x$ to $J_y$ are bounded.
\end{proof}

\begin{proof}[Proof of Theorem \ref{proa}] Theorem \ref{proa} is a variation of \cite[Theorem 2.1]{BP} where the hypotheses of no connection and no periodic orbit were weakened. Here we just point out which change is necessary in the proof of \cite[Theorem 2.1]{BP}. In this regard, \cite[Lemma 3.2]{BP} ought to be replaced by Lemma \ref{lemafour}.
The hypothesis that $f$ has no periodic orbit in the statement of \cite[Theorem 2.1]{BP} is not necessary: all we need is that $\omega(x)$ is  infinite. In this way, the claims of \cite[Theorem 2.1]{BP} hold in our context, which proves Theorem \ref{proa}.
\end{proof}

\section{Proof of Theorem \ref{thm1}}

Throughout this section, let $f{:}\,I\to I$ be an $n$-PC with associated partition $I_1,\ldots,I_n$ having endpoints
$0=x_0<x_1<\cdots< x_n=1$. We will need the following   elementary result.
 
\begin{lemma}[{\cite[Lemma 3.6]{NP}}]\label{lemma3.6}
There exist $r\le 2n$ pairwise disjoint open intervals $F_1,\ldots, F_r$  such that
$f^{k}({F_j})$, $1\le j\le r$, $k<0$, are empty sets, and $f^k(F_j)$, $1\le j\le r$, $k\ge 0$, are pairwise disjoint open intervals and $\Omega=\cup_{j=1}^r \cup_{k\ge 0} f^k(F_j)$
 is a dense subset of $I{\setminus}\{x_0,x_1,\ldots,x_{n-1}\}$ having Lebesgue measure $1$.
\end{lemma}

A non-empty compact subset $\Lambda\subset [0,1]$ is an {\it attractor} of $f$ if there exists 
$p\in I$ such that $\Lambda=\omega(p)$. Let $F_1,\ldots,F_r$ be as in the statement of Lemma \ref{lemma3.6}, then for each $1\le j\le r$, $\cup_{k\ge 0} f^k(F_j)\cap \{x_0,x_1,\ldots,x_{n-1}\}=\emptyset$, implying that $\omega(p_j)$ is the same for any $p_j\in\cup_{k\ge 0} f^k(F_j)$. In this way, the attractors
\begin{equation}\label{attractors} 
\Lambda_1=\omega(p_1),\ldots, \Lambda_r=\omega(p_r) 
\end{equation}
do not depend on the choice of $(p_1,\ldots,p_r)\in F_1\times\cdots\times F_r$. 
\begin{lemma}\label{34987} Let $p\in I$. If $\omega(p)$ is infinite, then $\omega(p)\in \Lambda_1\cup\cdots\cup\Lambda_r$.
\end{lemma}
\begin{proof} Since $\omega(p)$ is infinite, the $f$-orbit of $p$ is not periodic. In particular, there exists
$k_0\ge 0$ such that the $f$-orbit of $f^{k_0}(p)$ does not pass through discontinuities. By the density of $\Omega$, there exists $1\le j\le r$ such that $f^{k_0}(p)\in \omega(p_j)$. Then, $\omega(p)=\omega\big(f^{k_0}(p)\big)\subset \omega(p_j)=\Lambda_j$. 
\end{proof}
Without loss of generality, by replacing $r$ by a smaller number, we may assume that the sets
$\Lambda_1,\ldots,\Lambda_r$ are pairwise distinct. It follows from Lemma \ref{lemma3.6} that $S=I{\setminus}\Omega$ is a Lebesgue null set. Let $1\le j\le r$. As $\overline{S}=S\cup\{1\}$ and $\Lambda_j\subset\overline{S}$, we have that $\overline{S}$ has empty interior, hence $\Lambda_j$ is totally disconnected. By the Cantor-Bendixson Theorem, we conclude that $\Lambda_j$ is either a finite set or the union of a Cantor set with a discrete set. If all the attractors $\Lambda_1,\ldots,\Lambda_r$ are finite, then, by Lemmas \ref{lemaone} and \ref{34987}, all natural codings of $f$ are ultimately periodic and we are done. Otherwise, there are $1\le s\le r$ infinite attractors. Without loss of generality, assume that $\Lambda_1,\ldots,\Lambda_s$ are the infinite attractors.
It follows from Theorem \ref{proa} that for each $1\le j\le s$, there exists a non-atomic $f$-invariant Borel probability measure $\mu_j$ whose support is $\Lambda_j$. Hence, if
$$ \mu=\frac1s \mu_1 +\cdots+\frac1s \mu_s,\quad A=\Lambda_1\cup\cdots\cup \Lambda_s,$$ 
then $\mu$ is a non-atomic $f$-invariant Borel probability measure with support equal to $A$.

Let $h{:}\,[0,1]\to [0,1]$ be the nondecreasing, continuous, surjective map  defined by $h(t)=\mu\big( [0,t]\big)$, $t\in I$. Notice that
$h$ is strictly increasing on $A$ and  constant on each connected component of $I{\setminus}A$. Given $x,x'\in {I_i}$ with $h(x)=h(x')$, we claim that $h\big(f(x)\big)=h\big(f(x')\big)$. Since $f$ is injective, $f\vert_{{I}_i}$ is either increasing or decreasing. Without loss of generality, in what follows, assume that $f\vert_{I_i}$ is increasing (and continuous) for every $1\le i\le n$. Assume $x\le x'$, then 
$f(x)\le f(x')$. Moreover, since $f\vert_{I_i}$ is increasing and continuous,
 $$ [x,x')=f^{-1}\left(\big[f(x),f(x') \big) \right).$$
 Hence, since $\mu$ is non-atomic and $f$-invariant, we have that for any $x\le x'$ in $I_i$,
 \begin{equation}\label{1840}
 h\big(f(x') \big)-h\big(f(x) \big) =\mu\Big(\big[f(x),f(x')\big)\Big)=\mu\Big(f^{-1}\big(\big[f(x),f(x')\big)\big)\Big)=h(x')-h(x),
 \end{equation}
 which proves the claim.
 
We will use $(\ref{1840})$ to define an IET $T{:}\,I\to I$. Let
 $$\mathcal{I}=\big\{1\le i\le n: (x_{i-1},x_i)\cap A\neq\emptyset\big\},$$
 where $x_{i-1}$ and $x_i$ are the endpoints of $I_i$. Let $m\le n$ be the cardinality of $\mathcal{I}$, then we may write $\mathcal{I}=\{i_1,\ldots,i_{m}\}$. Let $0=y_0<y_1<\cdots<y_{m}=1$ be the points defined by $y_{\ell}=h(x_{i_{\ell}})$, $1\le \ell\le m$. Let $T{:}\,I\to I$ be the map that at $h(x)\in I{\setminus}\{y_0,y_1,\ldots,y_{m-1}\}$ takes the value
\begin{equation}\label{TTT}
 T\big(h(x)\big)=h\big(f(x)\big).
\end{equation}
The map $T$ is well-defined on $I{\setminus} \{y_0,y_1,\ldots,y_{m-1}\}$. To see that, let $x,x'\in I$, $x<x'$, be such that $h(x)=h(x')$ is a point in $I{\setminus} \{y_0,y_1,\ldots,y_{m-1}\}$. Then, $\{x,x'\}\subset\bigcup_{i\in\mathcal{I}} (x_{i-1},x_i)$, otherwise $x$ or $x'$ would belong to  $\{y_0,y_1,\ldots,y_{m-1}\}$. In this way, there exist $i,j\in\mathcal{I}$ such that
$x\in (x_{i-1},x_i)$ and $x'\in (x_{j-1},x_j)$. If $i\neq j$, then the hypothesis $h(x)=h(x')$ yields
$h(x)=h(x_{i})=h(x_{j-1})=h(x')$, showing that $h(x)\in \{y_0,y_1,\ldots,y_{m-1}\}$, which is a contradiction. Hence, the only alternative left is $i=j$ and $x,x'\in (x_{i-1},x_i)$. In this way, $x,x'$ belong to the same interval $I_i$ and $(\ref{1840})$ implies that $T\big(h(x)\big)=T\big(h(x')\big)$, thus $T$
 is well-defined on $I{\setminus} \{y_0,y_1,\ldots,y_{m-1}\}.$
 
Let us prove that $T\vert_{(y_{\ell-1},y_{\ell})}$, $1\le \ell\le m$, is a translation. If $y,y'$ are two points in $(y_{\ell-1},y_{\ell})$, there exist $x,x'\in (x_{i_{\ell}-1},x_{i_{\ell}})$ such that $y=h(x)$ and $y'=h(x')$, then $(\ref{1840})$ and $(\ref{TTT})$ yield $$
 T(y')-T(y)=T\big(h(x')\big)-T\big(h(x)\big)=h\big(f(x')\big)-h\big(f(x)\big)=h(x')-h(x)=y'-y,
 $$
 proving that $T\vert_{(y_{\ell-1},y_{\ell})}$ is a translation.  In particular, $T\vert_{(y_{\ell-1},y_{\ell})}$ is injective and 
 $T\big((y_{{\ell-1}},y_{{\ell}})\big)$ is an open interval for each $1\le \ell \le m$. Moreover, since $h$ is order-preserving, if
 $\ell\ne k$, then $h\big(x_{i_{\ell}-1},x_{i_{\ell}}\big)$ and $h\big(x_{i_{k}-1},x_{i_{k}}\big)$ are non-overlapping open intervals,
 implying that $T$ is (globally) injective on $I{\setminus} \{y_0,y_1,\ldots,y_{m}\}$. As for the definition of $T$ on the set $\{y_1,\ldots,y_{n-1}\}$, we can choose one of the lateral limits of $f$ as we approach each of these points in such a way that $T$ is, indeed, globally injective. In this way, $T$ is a $m$-IET. 
 
  We claim that $T$ has no attractive periodic orbit. In fact, if for each $1\le j\le s$, $\gamma_i$ is an infinite $f$-orbit dense in $\Lambda_i$, then the union of the  infinite $T$-orbits $T(\gamma_1)$, \ldots, $T(\gamma_s)$ is a dense subset of $I$, ruling out attractive periodic $T$-orbits.
 
Let $x\in I$ be a point whose natural $f$-coding is $\theta=\theta_0 \theta_1\ldots$, then we may assume that $\omega(x)$ 
is infinite, otherwise Lemma \ref{lemaone} says that $\theta$ would be ultimately periodic (i.e. $\exists k\ge 0$ such that the natural $f$-coding of $f^k(x)$ is periodic). By Corollary \ref{cor35}, there exists $k_0\ge 0$ such that $f^k(x)\in (x_{i_1-1},x_{i_1})\cup\cdots\cup (x_{i_m-1},x_{i_{m}})$ for all $k\ge k_0$. This means that the natural $f$-coding $\zeta=\zeta_0\zeta_1\ldots$ of
$f^{k_0}(x)$ is an infinite word over the alphabet $\mathcal{A}'=\{i_1,\ldots,i_{m}\}$. Let $\eta=\eta_0\eta_1\ldots$ be the natural $T$-coding
of $y=h(x)$, then $\zeta_j=i_\ell\in \{i_1,\ldots,i_{m}\}$ if and only if $\eta_j=\ell\in \{1,\ldots,m\}$, proving that $\zeta$ and $
\eta$ are isomorphic infinite words.

\section{Proofs of Theorem \ref{thm2} and Corollary \ref{cor2}}

\begin{proof}[Proof of Theorem \ref{thm2}] Let $T:I\to I$ be a topologically transitive $n$-IET and $J_1,\ldots,J_n$ be the associated partition. Without loss of generality we may assume that the endpoints of $J_i$ are $y_{i-1}$ and $y_i$, where $0=y_0<y_1<\cdots < y_n=1$. Let $\{p_k\}_{k=1}^\infty\subset I{\setminus}\{y_0,y_1,\ldots,y_{n-1}\}$ be a dense $T$-orbit. Given $k\ge 1$, let 
 \begin{equation}\label{G_k} \mathcal{L}_k=\{\ell\ge 1: p_\ell<p_k\}\quad\textrm{and}\quad
  G_k=\left[\sum_{\ell\in\mathcal{L}_k} 2^{-\ell},2^{-k}+\sum_{\ell\in\mathcal{L}_k} 2^{-\ell}\right].
 \end{equation}
 Notice that $p_k>0$ and $\mathcal{L}_k\neq\emptyset$. Hence, $G_k\subset (0,1)$ is a well-defined interval of length $\vert G_k\vert=2^{-k}$. We claim that $\{p_k\}_{k\ge 1}$ and $\{G_k\}_{k\ge 1}$ share the same ordering meaning that
 \begin{equation}\label{pk<pj}
 p_k<p_j\iff \sup G_k < \inf G_j.
 \end{equation}
 In fact, $p_k<p_j$ if and only if $ \{k\}\cup \mathcal{L}_k\subset \mathcal{L}_j$, which is equivalent to 
 $$\sup G_k=2^{-k}+\sum_{\ell\in\mathcal{L}_k} 2^{-\ell}<\sum_{\ell\in\mathcal{L}_j} 2^{-\ell}=\inf G_j.$$
 In particular, we have that the intervals $G_1,G_2,\ldots$ are pairwise disjoint and their union is dense because $\sum_{k=1}^\infty \vert G_k\vert=1$. Applying $(\ref{pk<pj})$ we conclude that if  $J\subset I$ is an interval and
 $$
 \{m_k\}_{k\ge 1}=\{\ell\ge 1: p_{\ell}\in J\},\quad \textrm{then}\,\, \overline{\cup_{k\ge 1} G_{m_k}}\quad\textrm{is an interval}. 
 $$
 
  Let $\widehat{h}{:}\,\cup_{k\ge1} G_k\to I$ be the function that on $G_k$ takes the constant value $p_k$. By $(\ref{pk<pj})$, we have that $\widehat{h}$ is nondecreasing and has  dense domain and dense range. Thus, $\widehat{h}$ admits a unique nondecreasing continuous surjective extension $h{:}\,[0,1]\to [0,1]$ to the whole interval $[0,1]$. It is elementary to see that
 $h^{-1}\big(\{p_k\}\big)={G_k}$. Denote by $I_1,\ldots,I_n$ the partition of $I$ defined by
 $I_i=h^{-1}(J_i)$. Notice that $x_i=h^{-1}(y_i)$, $0\le i\le n$, are the endpoints of the partition $I_1,\ldots,I_n$.
 
 Let $\widehat{f}:\cup_{k\ge 1} G_k \to \cup_{k\ge 2} G_k$ be such that $\widehat{f}\vert_{G_k}{:}\,G_k\to G_{k+1}$ is an  affine bijection with slope $\frac12 T'(p_k)$ for every $k\ge 1$, where $T'(p_k)\in \{-1,1\}$ is the derivative of $T$ at $p_k$. 
 We claim that for each $1\le i\le n$ , there exist a dense subset $\widehat{I}_i$ of $I_i$, $\lambda_i\in\left\{-\frac12,\frac12\right\}$ and $b_i\in\mathbb{R}$ such that
\begin{equation}\label{lxb}
\widehat{f}(x)=\lambda_i x+b_i\quad\textrm{for all}\quad x\in {\widehat I}_i.
\end{equation}
 In order to show that $(\ref{lxb})$ is true, fix $1\le i\le n$ and let $\{m_k\}_{k\ge 1}=\{\ell\ge 1:p_{\ell}\in J_i\}$,   
 then $\widehat{J_i}={\cup_{k\ge 1} \{p_{m_k}\}}$
  is a dense subset of $ {J_i}$ and  $\widehat{I_i}={\cup_{k\ge 1} G_{m_k}}$ is a dense subset of $I_i$. Moreover, there exists $\lambda_i\in\left\{-\frac12,\frac12\right\}$ such that $T'(y)=2\lambda_i$ for all $y\in J_i$. In particular, $T'(p_{m_k})=2\lambda_i$ for all $k\ge 1$. By definition, $\widehat{f}\vert_{G_{m_k}}{:}\,G_{m_k}\to G_{m_k+1}$ is an  affine bijection with slope $\frac12 T'(p_{m_k})=\lambda_i$ for all $k\ge 1$, which proves $(\ref{lxb})$. We have proved that there exist $\lambda_i\in\left\{-\frac12,\frac12\right\}$ and $c_{m_k}\in\mathbb{R}$ such that
  \begin{equation}\label{cmkcmk}
   \widehat{f}(x)=\lambda_i x + c_{m_k}\quad\textrm{for all}\quad x\in G_{m_k}.
  \end{equation}
 Let us prove that if $\lambda_i=\frac12$ $\big($respectively, $\lambda_i=-\frac12\big)$ then $\widehat{f}$ is strictly increasing (respectively, strictly decreasing) on $\cup_{k\ge 1} G_{m_k}$. Without loss of generality, assume that $\lambda_i=-\frac12$, then $\widehat{f}$ is strictly decreasing on each interval $G_{m_k}$. Let  $y_k<z_j$ be such that $y_k\in G_{m_k}$ and $z_j\in G_{m_j}$, where $k\neq j$ and $\sup G_{m_k}<\inf G_{m_j}$. By $(\ref{pk<pj})$, we have that  
  $p_{m_k}<p_{m_j}$ and $\{p_{m_k},p_{m_j}\}\subset J_i$. Then, since $T'(y)=2\lambda_i=-1$ for all $y\in J_i$, we have that $T\vert_{J_i}$ is decreasing, thus  $T(p_{m_k})>T(p_{m_j})$, that is,
  $p_{m_k+1}>p_{m_j+1}$. By $(\ref{pk<pj})$ once more, we get $\sup G_{m_j+1}<\inf G_{m_k+1}$.
   By definition, $f(y_k)\in G_{m_k+1}$ and $f(z_j)\in G_{m_j+1}$, thus $f(y_k)>f(z_j)$. This proves that $\widehat{f}$ is decreasing on $\cup_{k\ge 1} G_{m_k}$.  It remains to prove that $c_{m_k}$ in $(\ref{cmkcmk})$ is the same for all $k\ge 1$. Let $j\neq k$. We may assume that $a=\sup {G_{m_j}}<\inf G_{m_k}=b$. Notice that
   \begin{eqnarray*} 
   \frac12(b-a)+\frac{\vert\lambda_i\vert}{\lambda_i}(c_{m_k}-c_{m_j})&=&
    \frac{\vert\lambda_i\vert}{ \lambda_i}\big(\widehat{f}(b)-\widehat{f}(a)\big)= \sum_{G_{m_{\ell}}\subset [a,b]} \left\vert \widehat{f}\big(G_{m_{\ell}}\big)\right\vert \\ &=&\frac12 \sum_{G_{m_\ell}\subset [a,b]} |G_{m_{\ell}}|= \frac12 (b-a)
   \end{eqnarray*}
  yielding $c_{m_k}=c_{m_j}$. Thus, $(\ref{lxb})$ is true.
      
  It follows from $(\ref{lxb})$ that $\widehat{f}\vert_{\cup_{k\ge 1} G_{m_k}}$ admits a unique monotone continuous extension to the interval $I_i=h^{-1}(J_i)$. This extension is also an affine map with slope equal to $\frac12$ in absolute value. Since $i$ is arbitrary, we obtain an injective  piecewise $\frac12$-affine extension $f$ of $\widehat{f}$ to the whole interval $I=\cup_{i=1}^n I_i$. 
        
  It remains to show that $h\circ f=T\circ h$. In fact, for every $y\in G_k$, we have that 
\begin{equation}\label{eql}
h\big(f(y)\big)=\widehat{h}\big(\widehat{f}(y)\big)=p_{k+1}=T(p_k)=T\big(\widehat{h}(y)\big)=T\big(h(y)\big).
\end{equation}
Hence, $(\ref{eql})$ holds for a dense set of $y\in I$. By continuity, $(\ref{eql})$ holds for every $y\in I$.
\end{proof}

\begin{proof}[{Proof of Corollary \ref{cor2}}] Let $0<\alpha<1$ be irrational. Let $T{:}I\to I$ be the $2$-IET defined by $T(y)=y+\alpha\,\,({\rm mod}\,\,1)$, or equivalently, let $J_1=\left[0,1-\alpha\right)$, $J_2=\big[1-\alpha,1\big)$, and
$$T(y)=
\begin{cases} y+\alpha & \textrm{if}\quad y\in J_1 \\
y+\alpha -1 & \textrm{if}\quad y\in J_2 \\
\end{cases}.
$$
It is widely known that $T$ is minimal.
Hereafter, we take all the notation of the proof of Theorem $\ref{thm2}$. Let $y_0=0$, $y_1=1-\alpha$ and $y_2=1$.
Let $\gamma=\{p_k\}_{k=1}^{\infty}=\{\alpha,T(\alpha),\ldots\}$ be the $T$-orbit of $\alpha$, then $\gamma$ is a dense orbit contained in $I{\setminus}\{y_0,y_1\}$. Let $\theta=\theta_0 \theta_1\ldots$ be the natural $T$-coding of $\alpha$, then $\theta$ is a Sturmian word. Let us define the $2$-PC $f_T$. Let $G_k$, $k\ge 1$, be the pairwise disjoint intervals of length $\vert G_k\vert=2^{-k}$ defined by $(\ref{G_k})$. Let $I_i=h^{-1}(J_i)$ for $i=1,2$, then $I_1=[0,x_1)$, $I_2=[x_1,1)$, where $x_1=h^{-1}(y_1)$. Let 
$$\{m_k\}_{k\ge 1}=\{\ell\ge 1:p_{\ell}\in J_1\}=\{\ell\ge 1:\theta_{\ell-1}=1\},$$
 then $\widehat{J_1}={\cup_{k\ge 1} \{p_{m_k}\}}$
  is a dense subset of $ {J_1}$ and  $\widehat{I_1}={\cup_{k\ge 1} G_{m_k}}$ is a dense subset of $I_1$. In this way, since $\vert G_{m_k}\vert=2^{-m_k}$, we have that 
  $$ x_1=\sup I_1=\sum_{k\ge 1} \left\vert G_{m_k} \right\vert=\sum_{k\ge 1} 2^{-m_k}=\sum_{\ell\ge 1} (2-\theta_{\ell-1}) 2^{-\ell}=\frac12\sum_{\ell\ge 0} (2-\theta_\ell) 2^{-\ell}=2-\frac12\sum_{\ell\ge 0} \theta_{\ell} 2^{-\ell}.
  $$
  Since $T'(y)=1$ for every $y\in I$, we have that the slope $\lambda_i$ of $f_T$ is $\frac12$. In this way, we have that
  $$ f_T(x)=\begin{cases} \vspace{0.1cm} \dfrac12x + b_1& \textrm{if}\quad x\in [0,x_1) \\[0.1in]
  \dfrac12x + b_2& \textrm{if}\quad x\in [x_1,1) 
  \end{cases}.
  $$
  Since $\frac12 x_1+b_1=1$ and $\frac12x_1+b_2=0$, we conclude that
 \begin{equation*}\label{hide}
  f_T(x)=\frac12x + \delta,\quad \textrm{where}\quad \delta=\frac14\sum_{\ell\ge 0} \theta_{\ell} 2^{-\ell}.
  \end{equation*}
  It is clear that
  \begin{equation}\label{ddd}
  \delta=\frac14\left(1+\theta_0+\sum_{\ell\ge 1} (\theta_\ell-1)2^{-\ell}\right),
  \end{equation}
  thus $\{\theta_\ell-1\}_{\ell\ge 1}$ is the binary expansion of $\sum_{\ell\ge 1} (\theta_\ell-1)2^{-\ell}$. In this case, the transcendence of $\delta$ follows from Ferenczi and Mauduit \cite[Proposition 2]{FM} or Adamczewski and Cassaigne \cite[Theorem 1]{AC} together with the fact that $w=(\theta_1-1)(\theta_2-2)\ldots$ is a Sturmian word. 
  
 Now let us consider the particular case in which $\alpha=2-\varphi$, where $\varphi=(1+\sqrt{5})/2$ is the golden ratio. In this
 case, it is known that $(\theta_1-1)(\theta_2-1)\ldots$ is the Fibonacci word
 $$
 \theta-1=010010100100101001010010010100100101001010010010100\ldots
 $$
 The number 
 $$R=1-\sum_{\ell\ge 0} (\theta_{\ell}-1)2^{-(\ell+1)}=0.7098034428612913146\ldots$$
  is known in the mathematical literature as the rabbit constant. Notice that by $(\ref{ddd})$, we have that
  $$   \delta=\frac14\left(1+\theta_0+2\sum_{\ell\ge 1} (\theta_\ell-1)2^{-(\ell+1)}\right)=\frac14\left(2+2\sum_{\ell\ge 0} (\theta_\ell-1)2^{-(\ell+1)}\right)=1-\frac{R}{2}.
  $$
  The transcendence of the rabbit constant was proved by Davison \cite{D}.

 \end{proof}
 
 \section{Proofs of Corollary \ref{cor0} and Corollary \ref{cor1}}\vspace{0.5cm}

 \begin{proof}[Proof of Corollary \ref{cor0}]  
Let  $\theta=\theta_0\theta_1\ldots$ be a natural coding of  an injective $n$-PC $f{:}\,I\to I$. By Theorem \ref{thm1}, there exist a $m$-IET $T{:}\,I\to I$ with $2\le m\le n$ and $q\ge 0$ such that the infinite word $\theta^*=\theta_q\theta_{q+1}\ldots$ is either periodic or isomorphic to the non ultimately periodic natural $T$-coding $\omega=\omega_0 \omega_1\ldots$ of some point $y\in I$. For the sake of simplification, we will only consider the case in which
$T$ is an orientation-preserving $m$-IET with associated partition $I_1=[y_{0},y_1)$, $\ldots$, $I_m=[y_{m-1},y_m)$. Since $\omega$ is non ultimately periodic, there exist $r\ge 0$ and $y^*=T^r(y)$ whose $T$-orbit is regular, which means
 $$O_T(y^*)=\left\{T^r(y),T^{r+1}(y),\ldots\right\}\subset I{\setminus} \{y_0,y_1,\ldots,y_{m-1}\}.$$
 Because $O_T(y *)$ is regular, it is entirely contained in a minimal component of $T$. More specifically, there exist open intervals  $A_1,\ldots,A_p$ with pairwise disjoint closures such that $O_T(y^*)$ is a dense subset of $A_1\cup\cdots \cup A_p$ and
 $T(A_1)\subset A_{2}$, \ldots, $T(A_{p-1})\subset A_p$, $T(A_p)\subset A_1$, and $T$ takes $I{\setminus} (\overline{A_1}\cup\cdots\cup\overline{A_p})$ into itself
  (see \cite{MB2,NPT}).  Let $\mu$ be the normalized Lebesgue measure on $A_1\cup \cdots A_p$, then $\mu$ is $T$-invariant: $\mu\big(T^{-1}(B)\big)=\mu(B)$ for every Borel set $B\subset I$. Let $h{:}\,[0,1]\to [0,1]$ be the nondecreasing, continuous, surjective map  defined by $h(t)=\mu\big( [0,t]\big)$, $t\in I$. Notice that $h$ is strictly increasing on ${A}_1\cup \cdots\cup {A_p}$ and  constant on each of the finitely many connected components of $I{\setminus}{A}_1\cup \cdots\cup {A_p}$.
 Given $y,y'\in I_i$ with $h(y)=h(y')$, we claim that $h\big(T(y)\big)=h\big(T(y')\big)$. Without loss of generality, assume that $y\le y'$, then $T(y)\le T(y')$. Moreover, since $T\vert_{I_i}$ is a translation,
 $$ [y,y']=T^{-1}\left(\big[T(y),T(y') \big] \right).$$
 Hence, since $\mu$ is non-atomic and $T$-invariant, we have that for any $y,y'\in I_i$,
 \begin{equation}\label{1832}
 h\big(T(y') \big)-h\big(T(y) \big) =\mu\Big(\big[T(y),T(y')\big]\Big)=\mu\Big(T^{-1}\big(\big[T(y),T(y')\big]\big)\Big)=h(y')-h(y),
 \end{equation}
 which proves the claim.
 
We will use $(\ref{1832})$ to define an IET $E{:}\,I\to I$. Let
 $$\mathcal{I}=\big\{1\le i\le m:I_i\cap (A_1\cup \cdots\cup A_p)\neq\emptyset\big\}.$$
Let $m'\le m$ be the cardinality of $\mathcal{I}$, then we may write $\mathcal{I}=\{i_1,\ldots,i_{m'}\}$. The intervals $J_1=h\big(I_{i_1}\big)$, $\ldots$, $J_{m'}=h\big(I_{i_{m'}}\big)$ form a partition of $I$ into non-degenerate intervals with endpoints $0=z_0<z_1<\cdots<z_{m'}=1$ defined by $z_\ell=h(y_{i_\ell})$, $0\le \ell\le m'$. Let $E{:}I\to I$ be the right-continuous map that at $z=h(y)\in I{\setminus}\{z_0,z_1,\ldots,z_{m'-1}\}$ takes the value
\begin{equation}\label{EH321} 
E\big(h(y)\big)=h\big(T(y)\big).
\end{equation}
The map $E$ is well-defined. In fact, if $y,y'\in I$ are such that $h(y)=h(y')$, then $y,y'$ belong to the same connected component of $I{\setminus} (A_1\cup \cdots \cup A_p)$. There is no discontinuity of $T$ between $y$ and $y'$, otherwise $h(y)$ would belong to $\in \{z_1,\ldots,z_{m'-1}\}$. In this way, $y,y'$ belong to the same interval $I_i$ and $(\ref{1832})$ asserts that $E$ is well-defined. Notice that, by definition, $E(z_{\ell})=\lim_{\epsilon\to 0^+} E(z_{\ell}+\epsilon)$ for all $0\le \ell\le m'-1$.
 
Let us prove that $E\vert_{(z_{\ell-1},z_{\ell})}$, $1\le \ell\le m'$, is a translation. If $z,z'$ are two points in   $(z_{\ell-1},z_{\ell})$, then there exist $y,y' \in \big(y_{i_{\ell-1}},y_{i_\ell}\big)$ such that $z=h(y)$ and $z'=h(y')$. Now $(\ref{1832})$ and $(\ref{EH321})$ yield
 $$
 E(z')-E(z)=E\big(h(y')\big)-E\big(h(y)\big)=h\big(T(y')\big)-h\big(T(y)\big)=h(y')-h(y)=z'-z,
 $$
 proving that $E\vert_{J_\ell}$ is a translation. 
 
 The map $E$ is surjective. In fact, since $h$ and $T$ are surjective, given $z\in I$, there exists $y\in I$ such that
 $E\big(h(y)\big)=h\big(T(y)\big)=z$. To see that $E$ is also injective, by the above, $E$ takes each interval $J_\ell$ into its translate $E(J_\ell)$, which therefore has the same length, that is, $\left | E(J_\ell)\right |=\left | J_\ell\right |$. Since $E$ is surjective, we have that $$1=\sum_{\ell=1}^{m'} \left\vert E(J_\ell)\right\vert\le\sum_{\ell=1}^{m'} \left\vert J_\ell\right\vert\le 1,$$
implying that no overlapping is possible for the intervals $E(J_1),\ldots,E(J_{m'})$. This proves that $E$ is a $m'$-IET.

Becasuse $O_T(y^*)$ is a dense subset of $A_1\cup \cdots\cup A_p$ and $h(A_1\cup \cdots\cup A_p)$ is dense in $I$, we have that $h$ takes the $T$-orbit $O_T(y^*)$ onto a dense $E$-orbit, thus $E$ is topologically transitive. Moreover, if $\zeta=\zeta_0\zeta_1\ldots$ is the natural $T$-coding of $y^*$ and $\eta=\eta_0\eta_1\ldots$ is the natural $E$-coding
of $z^*=h(y^*)$, then $\zeta_k=i_\ell\in \{i_1,\ldots,i_{m'}\}$ if and only if $\eta_k=\ell\in \{1,\ldots,m'\}$, proving that $\zeta$ and $
\eta$ are isomorphic infinite words. To conclude the proof, we recall that $\theta_{q+r} \theta_{q+r+1}\ldots$ is isomorphic to  $\zeta$.

     \end{proof}  
     
 \begin{lemma}\label{sd} Let $\theta=\theta_0\theta_1\ldots$ be an infinite word  and $\theta^*=\theta_{q+1} \theta_{q+2}\ldots$ an infinite subword of $\theta$, then there exist   $k_0\ge 1$ and $\beta\ge 0$ such that
 $$
 p_k(\theta)=p_k(\theta^*)+\beta\quad\textrm{for every}\quad k\ge k_0.
 $$
 \end{lemma}
 \begin{proof} For each $k\ge q+1$, let
 \begin{eqnarray*}\mathcal{W}_k&=&\left\{\theta_0\theta_1\ldots\theta_{k-1},\quad \theta_1\theta_2\ldots\theta_{k},\quad \ldots,\quad\theta_q\theta_{q+1}\ldots\theta_{q+k-1}\right\}\subset L_k(\theta)
\\\mathcal{W}^*_k&=&\{\omega\in\mathcal{W}_k: \omega\not\in L_k(\theta^*)\}
\end{eqnarray*}
 Notice that $\mathcal{W}_k$ is formed by at most $q+1$ distinct finite words and
  $k\mapsto\#\mathcal{W}_k^*$ is a nondecreasing map, thus there exist $k_0\ge 0$ and $\beta\le q+1$ such that
  $\#\mathcal{W}_k=\beta$ for every $k\ge k_0$. Moreover, for every $k\ge k_0$, we have the disjoint union
  $$L_k(\theta)=\mathcal{W}_k^*{\cup} L_k(\theta^*),\quad\textrm{thus}\quad p_k(\theta)=p_k(\theta^*)+\beta.$$
  \end{proof}

 \begin{lemma}\label{lem47} Let $\theta$ be a natural coding of a topologically transitive $m$-IET $T{:}I\to I$, then there exist $k_0\ge 1$, $\alpha\in \{0,\ldots, m-1\}$ and  $\beta\ge 1$ such that 
 \begin{equation}\label{wr}
 p_\theta(k)=k\alpha+\beta\quad\textrm{for every}\quad k\ge k_0.
 \end{equation}
  Moreover, if $T$ is a standard $m$-IET, with $m\ge 2$, satisfying the 
  i.d.o.c., then $\alpha=m-1$, $\beta=1$ and $k_0=1$.
 \end{lemma}
 \begin{proof}  Let $T{:}\,I\to I$ be a topologically transitive $m$-IET and $\mathscr{P}=\{I_1,\ldots,I_m\}$ be the partition associated to $T$, then, since $T^{-1}$ is also an IET,
 $T^{-k}(\mathscr{P})$ is a partition of $I$ into  intervals for every $k\ge 0$, implying that the members of the set $$ \mathscr{P}_k=\bigwedge_{\ell=0}^{k-1} T^{-\ell}(\mathscr{P})=\left\{I_{i_0}\cap T^{-1}\big(I_{i_1}\big)\cap \cdots\cap T^{-(k-1)}\big(I_{i_{k-1}}\big): 1\le i_0,i_1,\ldots,i_{k-1}\le m\right\}. 
 $$
are pairwise disjoint intervals. Moreover, if $\theta$ is a natural coding of $T$, then the $k$-word 
$i_0 i_1\ldots i_{k-1}$ occurs in $\theta$ if and only if  the interval $J=I_{i_0}\cap T^{-1}\big(I_{i_1}\big)\cap \cdots\cap T^{-(k-1)}\big(I_{i_{k-1}}\big)\in \mathscr{P}_k$ is nom-empty.

 Let $\theta$ be the natural $T$-coding of some point $x\in I$. If $\theta$ is (ultimately) periodic, then by the Morse-Hedlund Theorem, there exist $k_0\ge 1$ and $\beta\ge 1$ such that $p_{\theta}(k)=\beta$ for every $k\ge k_0$, meaning that $(\ref{wr})$ holds with $\alpha=0$. Hence, we may assume that $\theta$ is not (ultimately) periodic. In this case,  there exists $q\ge 0$ such that the orbit $\{x^*,T(x^*),\ldots\}$ of $x^*=T^{q+1}(x)$ is a dense subset of $I{\setminus} \{x_0,x_1,\ldots,x_{m-1}\}$, where $0=x_0<x_1<\cdots <x_{m}=1$ are the endpoints of the partition $\mathscr{P}$. 
In this way, for each $k\ge 1$, $\{x^*,T(x^*),\ldots\}$ is contained in the union of the interiors of the intervals of $\mathscr{P}_k$. Hence, the $k$-word 
$i_0 i_1\ldots i_{k-1}$ occurs in the natural $T$-coding $\theta^*$ of $x^*$ if and only if  the interval $J=I_{i_0}\cap T^{-1}\big(I_{i_1}\big)\cap \cdots\cap T^{-(k-1)}\big(I_{i_{k-1}}\big)\in \mathscr{P}_k$ has non-empty interior. Therefore, the number of such intervals $J$ in $\mathscr{P}_k$
equals $p_k(\theta^*)$ and is related to the number of endpoints of the partition $\mathscr{P}_k$ as follows
\begin{equation}\label{pk}
 p_k(\theta^*)=1+\sum_{\ell=0}^{k-1} m_{\ell}, 
  \end{equation}
 where $m_0=m-1$ and
$$
  m_{\ell}=\left\{T^{-\ell}(x_1),\ldots,T^{-\ell}(x_{m-1})\right\}{\bigg\backslash} \bigcup_{p=0}^{\ell-1}\left\{T^{-p}(x_1),\ldots,T^{-p}(x_{m-1})\right\}
  $$
  gives the number of new division points at the $\ell$-th step towards the construction of $\mathscr{P}_k$. The map $\ell\mapsto m_{\ell}$ is a non-increasing, therefore there exist $k_0'\ge 0$ and $\alpha\ge 1$
  such that  $m_{\ell}=\alpha$ for every $\ell\ge k_0'$. Notice that $\alpha\ge 1$ because, as $\theta^*$ is not (ultimately) periodic, $p_k(\theta^*)\to\infty$ as $k\to\infty$.   Let $\beta_0,\beta_1,\ldots,\beta_{k_0-1}\ge 0$ be such that 
 \begin{equation}\label{nl}
 m_{\ell}=
 \begin{cases} \alpha+\beta_{\ell} & \textrm{if} \quad \ell \in \{0,1,\ldots,k_0'-1\}  \\
 \alpha &  \textrm{if} \quad \ell \ge k_0'
 \end{cases}.
 \end{equation} 
 By $(\ref{pk})$ and $(\ref{nl})$, we have that if $\beta'=1+\beta_0+\beta_1+\cdots+\beta_{k_0-1}$, then
  $$ p_k(\theta^*)=1+\sum_{\ell=0}^{k_0'-1} (\alpha+\beta_{\ell})+ \sum_{k_0'}^{k-1} \alpha= \alpha k + \beta' \quad\textrm{for all}\quad k\ge k_0'+1.$$
  By Lemma \ref{sd}, there exist $k_0\ge k_0'+1$ and $\beta''\ge 0$ such that
$$p_k(\theta)=p_k(\theta^*)+\beta''=\alpha k + (\underbrace{\beta'+\beta''}_{\beta})=\alpha k+\beta\quad\textrm{for all}\quad k\ge k_0.
$$
  Notice that if $T$ satisfies the i.d.o.c., then $\theta^*=\theta$ and $m_{\ell}=m-1$ for all $\ell\ge 0$, then $(\ref{pk})$ yields
  $$ p_k(\theta)=p_k(\theta^*)= (m-1) k +1 \quad \textrm{for all}\quad k\ge 1,
  $$  
  implying that in this case $(\ref{wr})$ holds with $\alpha=m-1$, $\beta=1$ and $k_0=1$. 
 \end{proof}
 
 \begin{proof}[Proof of Corollary \ref{cor1}] Let $f{:}I\to I$ be an injective $n$-PC and $\theta=\theta_0\theta_1\ldots$ be the natural $f$-coding of $x\in I$. By Corollary \ref{cor0}, there exist $k\ge 0$ and a topologically transitive $m$-IET, with $2\le m\le n$, such that the natural coding $\theta^*$ of $f^k(x)$ is either periodic or isomorphic to a non ultimately periodic natural coding of $T$. By Lemma \ref{lem47}, there exist $k_0\ge 1$, $\alpha\in \{0,\ldots, m-1\}$ and  $\beta\ge 1$ such that 
 \begin{equation}\label{a89}
 p_k(\theta^*)=k\alpha+\beta\quad\textrm{for all}\quad k\ge k_0.
 \end{equation}
 Notice that in the case in which $\theta^*$ is periodic, by the Morse-Hedlund Theorem,
 $(\ref{a89})$ holds with $\alpha=0$. To conclude the proof of the item $(i)$,  apply Lemma \ref{sd}. 
 As for tye item $(ii)$, we apply Theorem \ref{thm2} together with Lemma \ref{lem47}.
\end{proof}

\end{document}